\documentclass[12pt]{amsart}
\usepackage{amsthm}
\newtheorem{theorem}{Theorem}[section]
\newtheorem{lemma}[theorem]{Lemma}
\newtheorem{proposition}[theorem]{Proposition}
\theoremstyle{definition}
\newtheorem{definition}[theorem]{Definition}

\theoremstyle{remark}
\newtheorem{remark}[theorem]{Remark}
\newtheorem{corollary}[theorem]{Corollary}
\counterwithin*{section}{part}
\usepackage{amssymb}
\usepackage{empheq}
\usepackage{stmaryrd}
\usepackage{enumerate}
\usepackage[shortlabels]{enumitem}
\usepackage{calc}
\usepackage{url}
\providecommand{\dotdiv}{%
  \mathbin{%
    \vphantom{+}%
    \text{%
      \mathsurround=0pt
      \ooalign{%
        \noalign{\kern-.35ex}%
        \hidewidth$\smash{\cdot}$\hidewidth\cr
        \noalign{\kern.35ex}%
        $-$\cr
      }%
    }%
  }%
}
\usepackage{xcolor}
\usepackage{fonttable}
\DeclareFontFamily{U}{rcjhbltx}{}
\DeclareFontShape{U}{rcjhbltx}{m}{n}{<->s*[1.2]rcjhbltx}{}
\DeclareSymbolFont{hebrewletters}{U}{rcjhbltx}{m}{n}
\DeclareMathSymbol{\fe}{\mathord}{hebrewletters}{112}
\DeclareMathSymbol{\tsadi}{\mathord}{hebrewletters}{118}
\newcommand{\norm}[1]{\left\lVert#1\right\rVert}
\usepackage{mathtools}

\newcommand{\NN}{\mathbb{N}}

\newcommand{\RR}{\mathbb{R}}
\newcommand{\PP}{\mathbb{P}}
\newcommand{\EE}{\mathbb{E}}

\newcommand{\eps}{\varepsilon}

\usepackage[left=2.0cm,%
                right=2.0cm,%
                top=2.5cm,%
                bottom=3.5cm,%
                headheight=12pt,%
                a4paper]{geometry}%

\usepackage[colorlinks=true,urlcolor=blue,linkcolor=blue,plainpages=false,pdfpagelabels]{hyperref}

\begin{document}

\title[Quantitative limit theorems for generalized 
P\'olya urns]{Quantitative limit theorems for generalized 
P\'olya urns with applications to random tree models}

\author[Morenikeji Neri and Pedro Pinto]{Morenikeji Neri${}^{\MakeLowercase a}$ and Pedro Pinto${}^{\MakeLowercase b}$}
\maketitle

\vspace*{-5mm}
\begin{center}
{\scriptsize 
${}^a$ Department of Mathematics, Technische Universit\"at Darmstadt,\\
Schlossgartenstra\ss{}e 7, 64289 Darmstadt, Germany,\\ 
\protect{neri@mathematik.tu-darmstadt.de}\\
${}^b$ Center for Mathematical Studies, University of Lisbon,\\
Campo Grande, 1749-016 Lisboa, Portugal,\\
\protect{pmpinto@ciencias.ulisboa.pt}\\}
\end{center}

\begin{abstract}
We establish novel quantitative limit theorems for the asymptotic distribution 
of colours in a generalized P\'olya urn. Concretely, we construct explicit rates 
of convergence for the proportion of balls of each colour in the urn, both in 
square-mean and almost surely, under a general condition on the replacement 
matrix. As an application, we revisit three models of random recursive trees 
studied by Janson (Random Structures \& Algorithms 26 (2005), 69--83): random 
recursive trees, random plane recursive trees, and random recursive $d$-ary 
trees. For each model, we show that the corresponding outdegree statistics can 
be cast as generalized P\'olya urns, and thereby obtain explicit rates of 
convergence, both in $L^2$ and almost surely, for the proportion of nodes of 
each outdegree. In all three cases, the rates we obtain are of order $O(1/n)$ 
in $L^2$ and almost surely, and are uniform in the outdegree under consideration.
\end{abstract}
\section{Introduction}

We consider a generalized P\'olya urn process $(U(n))_{n\geq 0}$ defined as follows (cf.\ \cite{janson2004functional}).\footnote{This is a generalisation of the well-known process of P\'olya-Eggenberger urn~\cite{EggenbergerPolya1923} or Friedman urn~\cite{Friedman1949}.} One has an urn containing balls which can be one of $d\geq 2$ different colours and each $U(n)$ is a vector in $\NN^d$ stipulating the configuration $(U_0(n), \cdots, U_{d-1}(n))$ of the urn at time $n$, where $U_i(n)\geq 0$ indicates the number of balls of colour $i$. The urn starts in an initial state $U(0)\in \NN^d$ which is assumed not to be the zero vector. The dynamics of the urn are determined by a replacement matrix $(R_{ij})_{0\leq i,j\leq d-1}$, whose entries are integers. At each step, we uniformly at random select a ball from the urn and check its colour. If a ball of colour $i$ was selected, we add (or remove) balls to the urn according to the entries in the row $R_i$ (with negative entries corresponding to removing balls). 

Let $T(n):=\sum_{i=0}^{d-1}U_i(n)$ be the total number of balls inside the urn at time $n$. We say that a generalized P\'olya urn has \emph{balance} $\lambda \in \NN\setminus \{0\}$, if for all $0\leq i\leq d-1$ we have $\sum_{j=0}^{d-1} R_{ij}=\lambda$. A generalized P\'olya urn with a balance is said to be \emph{balanced}. This will imply that $\lambda$ is an eigenvalue of $R$. This means that the change in the number of balls in each step remains constant and therefore $T$ is deterministic, given by $T(n)=\sum_{i=0}^{d-1}U_i(0) + n\lambda$.\footnote{We assume in our generalized P\'olya urn model that the initial state $U(0)$ is always deterministically given.}

The removal of balls from the urn is permitted provided that the number of balls of every colour remains nonnegative throughout the evolution of the urn. Whenever the number of balls of a given colour becomes negative, the urn is declared extinct, and the process stops. Conditions on the replacement matrix $R$ and the initial composition $U(0)$ can be placed that guarantee that extinction almost surely never occurs (see e.g.\ \cite{Mahmoud2009}). One such condition is that for each $i$ there is an integer $d_i \ge 1$ such that $R_{ii}\ge 0$ or $R_{ii}=-d_i$ and $d_i$ divides $U_i(0)$ and $R_{ij}$ for all $0\le j\le d-1$. Such urns are called tenable \cite{bagchi1985asymptotic}.\\

Our main results are (quantitative) limit theorems which are akin to the classic \emph{strong law of large numbers} for generalized P\'olya urns:

\begin{theorem}[cf.\ Theorem 3.1 of \cite{janson2005asymptotic}]\label{urns-convergence}
    Let $(U(n))_{n\geq 0}$ be a tenable $d$-colour generalized P\'olya urn with a irreducible replacement matrix $R$, which is assumed to have balance $\lambda \in \NN\setminus\{0\}$. For $u$ an eigenvector of $R$ with eigenvalue $\lambda$ whose components sum to 1, we have
    \[
    \frac{U(n)}{n}\to \lambda u  \mbox{ almost surely.}
    \]
    Equivalently
     \[
    \hat{U}(n) := \frac{U_i(n)}{T(n)}\to u  \mbox{ almost surely.}
    \]
\end{theorem}
 Theorem \ref{urns-convergence} is typically attributed to Athreya and Karlin \cite{AthreyaKarlin1968}, who established the result under the assumption that the replacement matrix was nonnegative. The case when the urn is tenable and further extensions are due to Janson \cite{janson2004functional,janson2005asymptotic}.\\

 Our main results (Theorems \ref{Lemma-urns-slow} and \ref{Lemma-urns-fast}) establish a quantitative rendering of the convergence in Theorem \ref{urns-convergence} (both almost surely and in $L^2$) with the conditions that the urn is tenable and irreducible replaced with a different condition on the replacement matrix. To motivate our quantitative results on almost sure convergence, we recall that 
a stochastic process $(X_n)$ converges almost surely to a random variable $X$ 
if and only if, for all $\varepsilon > 0$,
\[
\lim_{n \to \infty} \PP\bigl(\exists\, m \geq n\, (|X_m - X| \geq \varepsilon)
\bigr) = 0. \tag{$\star$}
\]
This equivalence is entirely standard and follows from the continuity of 
probability measures (see, e.g., \cite[Lemma 3.1]{NeriPowell2024}). A natural way to quantify the speed of almost sure convergence is therefore to study the 
speed at which the probabilities in $(\star)$ tend to 0.\footnote{Together with Pischke and Oliva, the first author provided a proof-theoretic justification 
in \cite{NeriOlivaPischke2026} for why the computational content of almost sure 
convergence should be studied via $(\star)$, in the context of extending 
Kohlenbach's proof mining program \cite{Koh2008} to probability 
theory \cite{NeriPischke2024}.} This approach was used to study the speed of 
convergence in the Strong Law of Large Numbers by Siegmund \cite{siegmund1975large} 
(see also \cite{Fill:83:Largedev}) and more recently by the first 
author \cite{Neri2024,Neri2025}. There has also been a growing interest in 
using $(\star)$ to study quantitative almost sure convergence in nonlinear 
stochastic approximation and stochastic optimisation; 
see \cite{NeriPischkePowell2025b,NeriPischkePowell:fejer:rates,Pischke2025b,
Pischke2026,PischkePowell2024,PischkePowell2026} for example.\\

We illustrate our results with the following instance of our quantitative limit 
theorems. Throughout the paper, for a matrix $P$, we write $\operatorname{symm}(P):= (P+P^\top)/2$ 
for its symmetric part and $\lambda_{\max}(P)$ for its eigenvalue of largest real 
part.
 \begin{theorem}
 \label{thrm:main:app}
      Let $(U(n))_{n\geq 0}$ be $d$-colour generalized P\'olya urn with replacement matrix $R$, which is assumed to have balance $\lambda \in \NN\setminus\{0\}$ and satisfy 
      \[
A:=\lambda_{\max}(\operatorname{symm}((R - \lambda I) )) < -\frac{\lambda}{2}. 
\]
For $u$ an eigenvector of $R$ with eigenvalue $\lambda$ whose components sum to 1, we have
  \[
\EE\left[\norm{\hat{U}(n)-u}^2\right]\leq\frac{\lambda t}{\lambda(n+1)+T(0)} \mbox{ for all } n \in \NN\]
and 
\[
\PP\left(\exists m\geq n\left(\norm{\hat{U}(n)-u}\geq \varepsilon\right)\right)\leq \frac{\lambda}{\varepsilon^2}\cdot\frac{t+2q}{\lambda(n+1)+T(0)} \mbox{ for all } n \in \NN \mbox{ and }\eps>0
\]
where 
\[
t\geq \max\left\{\frac{-\lambda q}{2A+\lambda},5\left(1+\frac{T(0)}{\lambda}\right)\right\}, q:=\left(\frac{D+\lambda}{\lambda}\right)^2 \mbox{ and } D\ge \max_{i\le d-1}\norm{R_i}.
\]
 \end{theorem}
The proof of Theorem \ref{urns-convergence} given in \cite{janson2004functional,janson2005asymptotic} proceeds via an embedding argument that passes to continuous time and studies a related problem in that setting. We take a different approach: we first make the classical observation that the colour proportions of a generalised P\'olya urn satisfy a \emph{stochastic approximation algorithm}, and then study the quantitative convergence of such algorithms, relying on a similar approach of the first author, Pischke, and Powell, who establish in \cite{NeriPischkePowell:fejer:rates} quantitative results for general stochastic processes satisfying a monotonicity condition.\\

From Theorem \ref{urns-convergence}, Janson \cite{janson2005asymptotic} derived asymptotic results on the distribution of outdegrees in several random tree models. We show that these results are consequences of our limit theorems and thereby obtain novel quantitative strengthenings.

\subsection{Applications to random tree models}

A \emph{recursive} tree is a rooted tree such that each node is labelled by a natural number, and the labels of any path, starting from the root, form an increasing sequence. A \emph{random recursive tree} with $n$ nodes, $T_n$, is a random tree drawn uniformly from the $(n-1)!$ recursive trees with $n$ nodes. One can instead define random recursive trees algorithmically, where starting from a single node, we obtain $T_n$ from $T_{n-1}$, by attaching a new leaf to a node chosen uniformly at random from $T_{n-1}$.\\

Writing $X_i(n)$ for the number of nodes of outdegree $i \geq 0$ in a random recursive tree with $n$ nodes, Janson \cite{janson2005asymptotic} proved the following.

\begin{theorem}[{cf.\ Theorem 1.1 of \cite{janson2005asymptotic}}]
\label{thrm:jan}
    For all $i \geq 0$,
    \[
    \frac{X_i(n)}{n} \to \frac{1}{2^{i+1}} \quad \text{almost surely.}
    \]
\end{theorem}
This result extends earlier work of Mahmoud and Smythe \cite{mahmoud1992asymptotic}, who 
had treated nodes of outdegrees $0$, $1$, and $2$.\\

As an application to our main quantitative limit theorems, we obtain a quantitative version of this convergence, both almost surely and in $L^2$. Concretely, we obtain:
\begin{theorem}
\label{thrm:main:rec}
    For all $i \geq 0$, we have the following:
    \[
    \EE\left[\left|\frac{X_i(n)}{n} - \frac{1}{2^{i+1}}\right|^2\right] 
    \leq \frac{20}{n} \mbox{ for all } n \geq 1
    \]
    and,
    \[
    \PP\left(\exists\, m \geq n \left(\left|\frac{X_i(m)}{m} 
    - \frac{1}{2^{i+1}}\right| \geq \varepsilon\right)\right) 
    \leq \frac{108}{\varepsilon^2n} \mbox{ for all } \varepsilon >0 \mbox{ and } n > 2/\varepsilon.
    \]
\end{theorem}

In \cite[Section 3]{mahmoud1992asymptotic}, the authors compute the first and 
second moments of $X_i(n)$ explicitly for $i = 0, 1, 2$; from these calculations, 
one can verify that our $L^2$ bound is asymptotically sharp in these cases. Note that our Theorem \ref{thrm:main:rec} furthermore provides a qualitative improvement to Theorem \ref{thrm:jan} as the result demonstrates that the 
convergence of $X_i(n)/n$ is \emph{uniform in $i$}, both almost surely and in $L^2$.\\

We turn now to the second model studied in \cite{janson2005asymptotic}, that of 
the \emph{random plane recursive tree}. As before, a random plane recursive tree 
$T_n$ with $n$ nodes is obtained algorithmically by adding a leaf to $T_{n-1}$, but now the 
children of each node are ordered (from left to right, say), and the new leaf 
may be inserted in any gap between existing children, or at either end. A node 
of outdegree $d$ thus offers $d+1$ possible insertion points, so the total number 
of available positions in $T_{n-1}$ is $2(n-1) - 1$, and $T_n$ is obtained by 
choosing one of these positions uniformly at random. The outdegree distribution 
of random plane recursive trees was studied by Mahmoud, Smythe, and 
Szyma\'nski \cite{mahmoud1993structure}, and Janson \cite{janson2005asymptotic} subsequently obtained the following result. Writing 
$Y_i(n)$ for the number of nodes of outdegree $i \geq 0$ in a random plane 
recursive tree with $n$ nodes:

\begin{theorem}[{cf.\ Theorem 1.2 of \cite{janson2005asymptotic}}]
\label{thrm:jan:plane}
    For all $i \geq 0$,
    \[
    \frac{Y_i(n)}{n} \to \frac{4}{(i+1)(i+2)(i+3)} \quad \text{almost surely.}
    \]
\end{theorem}

We obtain the following quantitative strengthening.

\begin{theorem}
\label{thrm:main:plane}
    For all $i\geq 0$, we have the following:
    \[
    \EE\left[ \left| \frac{Y_i(n)}{n} - \frac{4}{(i+1)(i+2)(i+3)} \right|^2 \right]\leq \frac{272}{n} \text{ for all }\ n\geq 1
    \]
    and
    \[
    \PP\left( \exists m\geq  n \left( \left|\frac{Y_i(m)}{m} - \frac{4}{(i+1)(i+2)(i+3)}\right|\geq \eps \right) \right)\leq \frac{768}{\eps^2n} \text{ for all } \eps>0 \text{ and } n\geq 2/\eps.
    \]
\end{theorem}

Theorem \ref{thrm:main:plane} should be compared with \cite[Theorem 3]{mahmoud1993structure}, where the authors show
\[
\EE[Y_i(n)] = \frac{4n-2}{(i+1)(i+2)(i+3)} + O\!\left(\frac{1}{n}\right),
\]
which gives
\[
\EE\left[\frac{Y_i(n)}{n} - \frac{4}{(i+1)(i+2)(i+3)}\right] 
= O\!\left(\frac{1}{n}\right). \tag{$\star\star$}
\]
While the $L^2$ bound in Theorem~\ref{thrm:main:plane} yields only the weaker 
rate $O(1/\sqrt{n})$ for the quantity in~$(\star\star)$, it is a stronger result 
in two respects. First, it controls the speed of $L^2$ convergence directly, 
whereas $(\star\star)$ does not even imply $L^1$ convergence. Second, although the constant in \cite{mahmoud1993structure} for $(\star\star)$ is not given explicitly, an 
inspection of their proof reveals that it depends on the outdegree $i$; in contrast, our results demonstrate that, like the case of recursive trees, the convergence both almost surely and in $L^2$ is uniform in  $i$.\\

The final model studied by Janson \cite{janson2005asymptotic} is that of \emph{random recursive $d$-ary trees}, for $d \geq 2$. A random recursive $d$-ary tree with 
$n$ nodes is a tree drawn uniformly at random from the set of all subtrees with 
$n$ nodes of a complete infinite $d$-ary tree, rooted at the same root. In the case $d=2 $, the distribution of outdegrees for this random tree model was studied by Devroye \cite{devroye1991limit}. Fixing $d\ge 2$ and writing $Z_i(n)$ for the number of nodes of outdegree $i \geq 0$ in a random recursive $d$-ary recursive tree with $n$ nodes, Janson showed the following. 

\begin{theorem}[{cf.\ Theorem 1.2 of \cite{janson2005asymptotic}}]
\label{thrm:jan:dary}
    For all $i \geq 0$,
    \[
    \frac{Z_i(n)}{n} \to s_i \quad \text{almost surely.}
    \]
    Where 
    \[
    s_i:=\frac{\binom{2d-i-2}{d-2}}{\binom{2d-1}{d-1}}.
    \]
\end{theorem}
We provide the following quantitative strengthening of this result:
\begin{theorem}
\label{thrm:main:dary}
   For all $0\le i< d$, we have the following:
    \[
    \EE\left[ \left| \frac{Z_i(n)}{n} - s_i\right|^2 \right]\leq \frac{19d^2}{n} \text{ for all }\ n\geq 1
    \]
    and
    \[
    \PP\left( \exists m\geq  n \left( \left|\frac{Z_i(m)}{m} - s_d\right|\geq \eps \right) \right)\leq \frac{108d^2}{\eps^2n} \text{ for all } \eps>0 \text{ and } n\geq 2/\eps.
    \]
    Moreover, we have
        \[
    \EE\left[ \left| \frac{Z_d(n)}{n} - s_d\right|^2 \right]\leq \frac{19d^4}{n} \text{ for all }\ n\geq 1
    \]
    and
    \[
    \PP\left( \exists m\geq  n \left( \left|\frac{Z_d(m)}{m} - s_d\right|\geq \eps \right) \right)\leq \frac{108d^5}{\eps^2n} \text{ for all } \eps>0 \text{ and } n\geq 2/\eps.
    \]
\end{theorem}
We note that we did not attempt to optimise the constants used in the above results. We anticipate that much better constants can be achieved.

\section{The quantitative convergence of stochastic approximation algorithms}
\label{subsec:SA:quant}

In their most general form, \emph{stochastic approximation algorithms} are stochastic processes $(X_n)$ taking values in $\mathbb{R}^d$ and satisfying
\begin{equation*}
    X_{n+1} = X_n + \gamma_{n+1} V_{n+1},
\end{equation*}
where $(V_n)$ are random variables and $(\gamma_n)$ are step sizes, which may in general be stochastic but are deterministic in all cases considered here. Depending on the setting, additional assumptions are imposed to ensure that $(X_n)$ converges almost surely to a solution of the problem at hand.

Stochastic approximation algorithms originated in the 1950s with the seminal contributions of Robbins and Monro \cite{Robbins1951}, Kiefer and Wolfowitz \cite{kiefer-wolfowitz:52:scheme}, and Blum \cite{blum:54:stochastic}, where the central problem was to locate the roots of functions whose values can only be observed in the presence of noise. We refer the reader to the monograph \cite{Duflo97} for a thorough treatment of stochastic approximation algorithms and their applications, and to the survey \cite{pemantle2007survey} for examples arising specifically in the context of urn models.

The present paper, however, requires a quantitative version of the convergence of the following stochastic approximation algorithm. Here and throughout, for an $n \times n$ matrix $P$, we write $\langle \cdot,\cdot\rangle_P$ for the bilinear form on $\RR^n$ defined by $\langle x,y\rangle_P:= \langle xP,y\rangle$ (note that all bilinear forms on $\RR^n$ arise through such a $P$). Furthermore, one has that $\langle \cdot,\cdot\rangle_P$ is symmetric iff $P$ is symmetric. We write $\norm{x}_P:= \sqrt{|\langle x,x\rangle_P|}$ and note that $\norm{\cdot}_P$ is a semi-norm iff $P$ is semi-definite.
\begin{theorem}
\label{thrm:RM}
Let $(X_n)$ and $(\Delta M_n)$ be sequences of random variables, adapted to a filtration $(\mathcal{F}_n)$, taking values in $\RR^n$, $(\gamma_n)$ be a sequence of nonnegative real numbers, $F: \RR^n \to \RR^n$ be a measurable function on $\RR^n$ and $P$ be a symmetric, positive-semi-definite $n \times n$ matrix. Suppose:
\begin{itemize}
    \item [(i)] For all $n \ge 0$ we have $ X_{n+1}= X_n + \gamma_{n}(F(X_n)+\Delta M_{n+1})$
    \item [(ii)] $\EE[\Delta M_{n+1}\mid \mathcal{F}_n]=0$
    \item[(iii)]  For all $n \in \NN$, we have $\langle F(X_n),X_n -z\rangle_P \le-A\norm{X_n -z}_P^2$ for some  $z \in \RR^n$ and $A > 0$.
    \item[(iv)] For all $n \ge 0$ we have $\EE[\norm{F(X_n)+\Delta M_{n+1}}_P^2 \mid \mathcal{F}_n] \le K$ for some $K>0$
    \item[(v)] $\sum_{n=0}^\infty \gamma_n = \infty$ and $\sum_{n=0}^\infty \gamma_n^2 < \infty$.
 \item [(vi)] $\EE[\norm{X_0 - z}_P^2] < L$ for some $L>0$
\end{itemize}
Then $X_n$ converges to $z$, with respect to the semi-norm $\norm{\cdot}_P$, almost surely and in $L^2$. In other words
\[
\EE\left[\norm{X_n-z}^2_P\right] \to 0 \mbox{ and } \norm{X_n-z}_P \to 0 \mbox{ almost surely.}
\]
\end{theorem}
Rather than proving Theorem \ref{thrm:RM} directly, we establish the following quantitative strengthening, in which we obtain rates of almost sure convergence and $L^2$ convergence for $(X_n)$, expressed in terms of a suitable quantitative reading of condition $(v)$.
\begin{definition}
    Let $(x_n)$ be a sequence of nonnegative real numbers.
    \begin{itemize}
        \item If $\sum_{i=0}^\infty x_n = \infty$, we say a function $\theta: \NN \times \RR^+ \to \NN$ is a rate of divergence for $\sum_{i=0}^\infty x_n$ if $\sum_{n=k}^{\theta(k,b)}x_n\geq b$ for any $b>0$ and $k \in \NN$.
        
        \item If $\sum_{i=0}^\infty x_n <  \infty$, we say a function $\chi: \RR^+ \to \NN $ is a rate of convergence for $\sum_{i=0}^\infty x_n$ if $\sum_{n=\chi(\varepsilon)}^{\infty}x_n < \varepsilon$ for any $\varepsilon>0$.
    \end{itemize}
\end{definition}

\begin{theorem}
\label{res:RM:quant}
Let $(X_n)$ and $(\Delta M_n)$ be sequences of random variables, adapted to a filtration $(\mathcal{F}_n)$, taking values in $\RR^n$, $(\gamma_n)$ be a sequence of nonnegative real numbers, $F: \RR^n \to \RR^n$ be a measurable function on $\RR^n$ and $P$ be a symmetric, positive-semi-definite $n \times n$ matrix. Suppose conditions (i)-(iv) and (vi) of Theorem \ref{thrm:RM} hold and 
\begin{itemize}
    \item [(v')] $\sum_{n=0}^\infty \gamma_n=\infty$ with rate of divergence $\theta$ and $\sum_{n=0}^\infty \gamma_n^2< \Gamma$, for $\Gamma >0$, with rate of convergence $\chi$.
\end{itemize}
Then $X_n$ converges to $z$, with respect to the semi-norm $\norm{\cdot}_P$ in $L^2$ and, quantitatively,
\[
\forall\, \varepsilon > 0\ \exists\, n \geq \rho(\varepsilon) 
\left(\EE[\|X_n - z\|_P^2] < \varepsilon\right)
\]
with rate
\[
\rho(\varepsilon) := \theta\!\left(\chi\!\left(\frac{\varepsilon}{2K}\right),
\frac{L + K\Gamma}{A\varepsilon}\right).
\]
Moreover, $X_n$ converges to $z$, with respect to the semi-norm $\norm{\cdot}_P$, almost surely and, quantitatively,
\[
\forall\, \varepsilon, \delta > 0\ \exists\, n \geq \rho'(\delta, \varepsilon)
\left(\PP\left(\exists\, m \geq n \left(\|X_m - z\|_P \geq \varepsilon\right)
\right) < \delta\right)
\]
with rate $\rho'(\delta, \varepsilon) := \rho(\delta\varepsilon^2)$.
\end{theorem}
\begin{proof}
    Let $V(n):= \norm{X_n -z}_P^2$. Assumption (i) and the fact that $P$ is symmetric and positive-semi-definite imply
    \[
    V_{n+1} = V_n + 2\gamma_{n}\langle F(X_n)+\Delta M_{n+1}, X_n -z \rangle_P +\gamma_{n}^2\norm{(F(X_n)+\Delta M_{n+1})}_P^2.
    \]
    Taking conditional expectations of both sides and using assumptions (ii)-(iv) yields:
    \[
     \EE[V_{n+1}\mid \mathcal{F}_n] \le V_n -2A\gamma_{n}V_n 
     +K\gamma_n^2.\tag{$\dagger$}
    \]
Concretely, we use standard properties of the conditional expectation and condition (ii) to obtain
\[
\EE[\langle F(X_n)+\Delta M_{n+1}, X_n -z \rangle_P \mid \mathcal{F}_n] = \langle F(X_n), X_n -z \rangle_P
\]
noting that $X_n-z$ and $F(X_n)$ are $\mathcal{F}_n$-measurable random variables, then the other conditions are used to obtain the bounds.\\ 

We now argue as in \cite[Theorem 3.4]{NeriPischkePowell:fejer:rates}. Taking expectations in $(\dagger)$ and summing, yields that for all $n \in \NN$
    \[
    \sum_{i=0}^n2A\gamma_{i}\EE[V_{i}] \le \EE[V_0] - \EE[V_{n+1}] + \sum_{i=0}^nK\gamma_n^2 < L +K\Gamma.
    \]
Now we have 
\[
\forall\varepsilon>0\ \forall N\in\mathbb{N}\ \exists n\in \left[N;\theta\left(N,\frac{L +K\Gamma}{2A\varepsilon}\right)\right](\EE[V_n]<\varepsilon)
\]
as if not, we may take $\varepsilon>0$ and $N\in\NN$ with $\EE[V_n]\ge\varepsilon$ for all $n\in \left[N;\theta\left(N,\frac{L +K\Gamma}{2A\varepsilon}\right)\right]$. This would entail
\[
\frac{L +K\Gamma}{2A}\leq \varepsilon\sum_{n=N}^{\theta(N,(L+K\Gamma)/2A\varepsilon)} \gamma_n\le \sum_{n=N}^{\theta(N,(L+K\Gamma)/2A\varepsilon)} \gamma_n\EE[V_n]\leq \sum_{n=0}^{\infty} \gamma_n\EE[V_n]<\frac{L +K\Gamma}{2A}
\]
which is a contradiction. For any $n\in\mathbb{N}$, define
\[
U_{n}:=V_{n}+\sum_{i=n}^\infty K\gamma_i^2.
\]
$(U_{n})$ is a nonnegative supermartingale, since for any $n\in\mathbb{N}$ we have
\begin{align*}
\EE[U_{n+1}\mid\mathcal{F}_n]=\EE\left[V_{n+1}\mid\mathcal{F}_n\right]+\sum_{i=n+1}^\infty K\gamma_i^2
\leq V_n + K\gamma_n^2+\sum_{i=n+1}^\infty K\gamma_i^2
=U_{n}.
\end{align*}
Now, Let $\varepsilon >0$ be given and suppose $m \ge \rho(\varepsilon)$. We can take 
\[
n(\varepsilon)\in \left[\chi\left(\frac{\varepsilon}{2K}\right);\theta\left(\chi\left(\frac{\varepsilon}{2K}\right),\frac{L +K\Gamma}{A\varepsilon}\right)\right]
\]
such that $\EE[V_{n(\varepsilon)}]\le \varepsilon/2$, in particular, $m \ge n(\varepsilon)$. Thus, As $(U_n)$ is a supermartingale, we have $\EE[V_m] \le \EE[U_{n(\varepsilon)}] = \EE[V_{n(\varepsilon)}]+\sum_{i=n(\varepsilon)}^\infty K\gamma_i^2 \le \EE[V_{n(\varepsilon)}]+\sum_{i=\chi(\varepsilon/2K)}^\infty K\gamma_i^2\le \varepsilon$. This yields that $\rho$ is a rate of convergence for $\EE[V_n] \to 0$. For the rate that $V_n\to 0$ a.s., note that
\begin{align*}
\PP(\exists m\geq \rho(\delta\varepsilon^2)(\norm{X_m-z}\geq \varepsilon))& \leq \PP(\exists m\geq \rho(\delta\varepsilon^2)(V_m\geq \varepsilon^2))\\&
\leq \PP(\exists m\geq n(\delta\varepsilon^2)(U_{m}\geq \varepsilon^2))\\
&\leq \frac{\EE[U_{n(\delta\varepsilon^2)}]}{\varepsilon^2} \le \delta,
\end{align*}
where the first inequality follows from the definition of $V_n$, the second inequality from the fact that $\rho(\delta\varepsilon^2) \ge n(\delta\varepsilon^2)$ and $U_n \ge V_n$ and the last inequality follows from Ville's inequality, for nonnegative supermartingales. This then implies that $\norm{X_m-z}\to 0$ a.s.\ (or equivalently that $X_n$ converges to $z$ a.s.\ ) with rate $\rho'$.
\end{proof}
For suitable parameters of the stochastic approximation algorithms in Theorem \ref{res:RM:quant}, one is able to obtain \emph{fast rates}. We first need a lemma from \cite{NeriPischkePowell:fejer:rates}.
\begin{lemma}[cf.\ Lemma 3.5 of \cite{NeriPischkePowell:fejer:rates}]
\label{lem:ssLikeForQM}
Suppose that $(x_n)$ is a sequence of nonnegative reals such that for $c>1$, $q\geq 0$ and $r>0$, we have
\[
x_{n+1}\leq \left(1-\frac{c}{n+r}\right)x_n+\frac{q}{(n+r)^2}
\]
for all $n\in\NN$. Then for all $n\in\NN$:
\[
x_n\leq \frac{t}{n+r} \ \text{ for }\ t\geq \max\left\{\frac{q}{c-1},rx_0\right\}.
\]
\end{lemma}

\begin{proof}
We show the claim by induction. The case $n=0$ holds by the definition $t$, and for the induction step, we note that since $q\leq t(c-1)$, we have $q\leq t\left(c-\frac{n+r}{n+r+1}\right)$ for all $n,r$, and therefore
\begin{align*}
x_{n+1}&\leq \left(1-\frac{c}{n+r}\right)\frac{t}{n+r}+\frac{q}{(n+r)^2}\\
&\leq t\left(\left(1-\frac{c}{n+r}\right)\frac{1}{n+r}+\left(c-\frac{n+r}{n+r+1}\right)\frac{1}{(n+r)^2}\right)\\
&=t\left(\frac{1}{n+r}-\frac{1}{(n+r+1)(n+r)}\right)\\
&=\frac{t}{n+r+1}
\end{align*}
which completes the induction.
\end{proof}
We note that in \cite{NeriPischkePowell:fejer:rates} the above lemma was stated under the additional assumption that  $r\in \NN$; however, as the proof presented here shows, this assumption is not needed. We now have the following result, which yields fast rates for the convergence of the stochastic approximation algorithm $(X_n)$ in Theorem \ref{thrm:RM}, under additional assumptions imposed on the parameters $\gamma_n$.
\begin{theorem}
\label{res:RM:quant:fast}
Let $(X_n)$ and $(\Delta M_n)$ be  sequences of random variables, adapted to a filtration $(\mathcal{F}_n)$, taking values in $\RR^n$, let $(\gamma_n)$ be a sequence of nonnegative real numbers, and let $F: \RR^n \to \RR^n$ be a measurable function on $\RR^n$ and $P$ be a symmetric, positive-semi-definite $n \times n$ matrix. Suppose (i)-(iv) and (vi) of Theorem \ref{thrm:RM} hold, and 
\begin{itemize}
    \item [(v'')] $K\gamma_n^2\leq q/(n+r)^2$ and $c/(n+r)\leq 2A\gamma_n$, for all $n\in\mathbb{N}$ where $c>1$, $q\geq 0$ and $r\ge1$.
\end{itemize}
 Then
\[
\EE[\norm{X_n-z}_P^2]\leq\frac{t}{n+r} \text{ and }\PP\left(\exists m\geq n(\norm{X_n-z}_P\geq \varepsilon)\right)\leq \frac{t+2q}{\varepsilon^2(n+r)}
\]
for all $n\in\mathbb{N}$ and $\varepsilon>0$, where $t\geq \max\left\{q/(c-1),rL\right\}$.
\end{theorem}
\begin{proof}
 From $(\dagger)$ in the proof of Theorem \ref{res:RM:quant}, for $V_n:=\norm{X_n-z}_P^2$, we have
       \[
     \EE[V_{n+1}\mid \mathcal{F}_n] \le V_n -2A\gamma_{n}V_n 
     +K\gamma_n^2.
    \]

We now argue as in \cite[Theorem 3.6]{NeriPischkePowell:fejer:rates}. Taking expectations $(\dagger)$ yields
\begin{align*}
\EE[V_{n+1}]&\leq \EE[V_n]-2A\gamma_{n}\EE[V_n]+K\gamma_n^2\\
&\leq (1-2A\gamma_{n})\EE[V_n]+K\gamma_n^2\\
&\leq \left(1-\frac{c}{n+r}\right)\EE[V_n]+\frac{q}{(n+r)^2}.
\end{align*}
Applying Lemma \ref{lem:ssLikeForQM} yields that for all $n\in\NN$
\[
\EE[V_n]\leq \frac{t}{n+r} \ \text{ for }\ t\geq \max\left\{\frac{q}{c-1},rL\right\} 
\]
so we have the first claim. For the second claim, first note that, using the assumption that $r \ge 1$, we have for all $i\ge 0$
\[
K\gamma_i^2\le\frac{q}{(i+r)^2}\leq \frac{2q}{(i+r)(i+r+1)}=2q\left(\frac{1}{i+r}-\frac{1}{i+r+1}\right)
\]
so that summing yields that, for all $n \ge 0$
\[
\sum_{i=n}^\infty K\gamma_i^2\leq 2q\sum_{i=n}^\infty\left(\frac{1}{i+r}-\frac{1}{i+r+1}\right)=\frac{2q}{n+r}.
\]
Now define 
\[
U_{n}:=V_n+\sum_{i=n}^\infty K\gamma_i^2.
\]
As in the proof of Theorem~\ref{res:RM:quant}, we have that $(U_{n})$ is a nonnegative supermartingale. Now, we have
\[
\EE[U_n]=\EE[V_n]+\sum_{i=n}^\infty K\gamma_i^2\leq \frac{t+2q}{n+r},
\]
and Ville's inequality yields
\[
\PP\left(\exists m\geq n(V_m\geq \varepsilon)\right)\leq \PP\left(\exists m\geq n(U_m\geq \varepsilon)\right)\leq \frac{\EE[U_n]}{\varepsilon}\leq \frac{t+2q}{\varepsilon(n+r)}
\]
and the second claim follows.
\end{proof}

\section{Main results: Quantitative limit theorems for generalized P\'olya urns }
\label{subsec:urns}

Throughout this section we fix a $d$-colour generalized P\'olya urn $(U(n))_{n\geq 0}$ with replacement matrix $R$, which is assumed have balance $\lambda \in \NN\setminus\{0\}$. The evolution of the urn is governed recursively by
\[
U(n+1)=U(n) + R_{\xi(n+1)},
\]
where $\xi(n+1)$ are random variables taking values in $\{0, \cdots, d-1\}$ with distribution
\[
\mathbb{P}\!\left(\xi(n+1) = i \mid U(n)\right)
    = \frac{U_i(n)}{T(n)}=:\hat{U}_i(n) \,
    \qquad \forall\, 0 \leq i \leq d-1.
\]
We recall the classical observation that the colour proportions $(\hat{U}(n))_{n\geq 0}$ evolve as a stochastic approximation algorithm (see, for example, Section 2.4 of the survey by Pemantle \cite{pemantle2007survey} or Lemma 2.9 of the lecture notes of Mailler \cite{mailler:urn:notes}). Throughout this section, $(\mathcal{F}_n)$ denotes the natural filtration generated by $(U(n))$.
\begin{proposition}\label{prop:urn:SA}
For all $n \geq 0$,
\[
    \hat{U}(n+1) = \hat{U}(n) + \gamma_n\bigl(F(\hat{U}(n)) + \Delta M_{n+1}\bigr),
\]
where $\gamma_n = \frac{1}{T(n+1)}$, $ \Delta M_{n+1}
    := R_{\xi(n+1)} 
      - \mathbb{E}\left[R_{\xi(n+1)} \mid \mathcal{F}_n\right]$, and  $F: \RR^d \to \RR^d$ is a function given by $F(x) = \sum_{i=0}^{d-1} x_i\bigl(R_i - \lambda\, x\bigr)$. Moreover, on the simplex
\[
    \mathcal{S} = \Bigl\{(x_0, \ldots, x_{d-1}) \in [0,1]^d : \sum_{i=0}^{d-1} x_i = 1\Bigr\},
\]
we have $F(x)= x(R-\lambda I)$ which implies $F(\hat{U}(n)) = \hat{U}(n)(R-\lambda I)$.

\end{proposition}
 
\begin{proof}
From the recursive definition of $(U(n))_{n \geq 0}$, we have
\begin{align*}
    \hat{U}(n+1)
    &= \frac{U(n) +R_{\xi(n+1)}}{T(n+1)}
     = \frac{U(n)}{T(n)} \cdot \frac{T(n)}{T(n+1)} + \frac{R_{\xi(n+1)}}{T(n+1)} \\[6pt]
    &= \hat{U}(n) \cdot \left(1 - \frac{\lambda}{T(n+1)}\right)
       + \frac{R_{\xi(n+1)}}{T(n+1)}
       =\hat{U}(n)+ \frac{R_{\xi(n+1)} - \lambda\,\hat{U}(n)}{T(n+1)}.
\end{align*}
Now, we set
\[
    \Delta M_{n+1}
    = R_{\xi(n+1)} - \lambda\,\hat{U}(n)
      - \mathbb{E}\!\left[R_{\xi(n+1)} - \lambda\,\hat{U}(n)\mid \mathcal{F}_n\right]=R_{\xi(n+1)} 
      - \mathbb{E}\!\left[R_{\xi(n+1)} \mid \mathcal{F}_n\right].
\]
Since $\mathbb{P}(\xi(n+1) = i \mid \mathcal{F}_n) = \hat{U}_i(n)$, we get that
\[
    \mathbb{E}\!\left[R_{\xi(n+1)} - \lambda\,\hat{U}(n)
    \mid \mathcal{F}_n\right]
    = \sum_{i=1}^{d} \hat{U}_i(n)\bigl(R_i - \lambda\,\hat{U}(n)\bigr)=F(\hat{U}(n)).
\]
It thus follows that,
\[
    \hat{U}(n+1) = \hat{U}(n) + \frac{1}{T(n+1)}\bigl(F(\hat{U}(n)) + \Delta M_{n+1}\bigr).
\]
For $x \in \mathcal{S}$, we easily see for all $0\leq j\leq d-1$
\begin{align*}
    F(x)_j&= \sum_{i=0}^{d-1}x_iR_{ij}-\lambda x_j\sum_{i=0}^{d-1}x_i\\
    &=\sum_{i=0}^{d-1}x_iR_{ij}-\lambda x_j=(x(R-\lambda I)_j.
\end{align*}
\end{proof}
Now, if in addition, we have $u \in \mathcal{S}$ satisfying
\begin{equation*}
\langle F(x),x-u\rangle = \langle x(R-\lambda I),x-u\rangle  <0 \text{ for all } x\in \mathcal{S}\setminus\{u\}, \tag{$*$}
\end{equation*}
then Proposition \ref{prop:urn:SA}, together with standard convergence results for stochastic approximation algorithms (see, for example, \cite[Theorem 1.4.26]{Duflo97}), would imply that $\hat{U}(n)\to u$ almost surely. However, it is not clear that condition $(*)$ holds in general, even with the assumptions that the urn is tenable and irreducible, although Theorem~\ref{urns-convergence} nevertheless guarantees that $\hat{U}(n)\to u$ almost surely in that case.

Nevertheless, it is still possible to obtain quantitative convergence results under additional assumptions closely related to the condition $(*)$, without the assumption of the urn being irreducible or tenable. In fact, we get a more general result, where we obtain convergence with respect to a semi-norm arising from a bilinear form. Recall that for a matrix $P$, we write $\langle x,y\rangle_P:= \langle xP,y\rangle$ and  $\norm{x}_P:= \sqrt{|\langle x,x\rangle_P|}$. Let us write $\norm{P}$ for the operator norm of $P$.
\begin{theorem}\label{Lemma-urns-slow}
  Let $P$ be a symmetric, positive-semi-definite $d \times d$ matrix. Assume that for all $n\in\NN$
  \begin{equation*}
    \langle F(\hat{U}(n)),\hat{U}(n) -u\rangle_P \le-A\norm{\hat{U}(n) -u}_P^2, \text{ for some } u \in \mathcal{S} \text{ and } A > 0.\tag{$**$}
\end{equation*}
Then $\hat{U}(n)$ converges to $u$, with respect to the semi-norm $\norm{\cdot}_P$ in $L^2$ and, quantitatively,
\[
\forall\, \varepsilon > 0\ \exists\, n \geq \rho(\varepsilon) 
\left(\EE[\|\hat{U}(n) - u\|_P^2] < \varepsilon\right)
\]
with
\[
\rho(\varepsilon) := \left\lceil\left(\left\lceil \frac{\ell}{\eps}\right\rceil +T(0)+1\right)\exp\left(\frac{\lambda(5\norm{P}+\ell)}{A\eps}\right)\right\rceil,
\]
where $\ell:=2\left(\frac{D+\lambda\sqrt{\norm{P}}}{\lambda}\right)^2$, $D$ is any number with $D\ge \max_{i\le d-1}\norm{R_i}$ and $T(0):=\sum_{i=0}^{d-1}U_i(0)$ denotes the initial total number of balls in the urn. Moreover, $\hat{U}(n)$ converges to $u$ almost surely, with respect to the semi-norm $\norm{\cdot}_P$ and, quantitatively,
\[
\forall\, \varepsilon, \delta > 0\ \exists\, n \geq \rho'(\delta, \varepsilon)
\left(\PP\left(\exists\, m \geq n \left(\|\hat{U}(m) - u\|_P \geq \varepsilon\right)
\right) < \delta\right)
\]
with rate $\rho'(\delta, \varepsilon) := \rho(\delta\varepsilon^2)$.
\end{theorem}

\begin{proof}
   The goal is to apply Theorem~\ref{res:RM:quant} with $F(x)= x(R-\lambda I)$. By Proposition \ref{prop:urn:SA}, conditions (i) and (ii) are satisfied. The assumption $(**)$ of the theorem corresponds to condition (iii).  Now, recall that
   \[
   F(\hat{U}(n))+\Delta M_{n+1}=R_{\xi(n+1)} - \lambda \hat{U}(n),
   \]
   and since,
   \[
   \|\hat{U}(n)\|_P = \sqrt{\langle \hat{U}(n)P, \hat{U}(n)\rangle} \le \sqrt{\norm{\hat{U}(n)P}\norm{\hat{U}(n)}} \le \sqrt{\norm{P}}
   \]
   where we use the fact that $\norm{\hat{U}(n)} \le 1$ to obtain the final inequality, we get
   \[
   \EE\left[ \left\| F(\hat{U}(n))+\Delta M_{n+1} \right\|_P^2 \mid \mathcal{F}_n \right] \leq (D+\lambda\sqrt{\norm{P}})^2.
   \]
   Hence, condition (iv) holds with $K=(D+\lambda\sqrt{\norm{P}})^2$. As $u\in \mathcal{S}$, we have $\norm{u}\le 1$. So, $\|\hat{U}(0)-u\|_P^2\leq \left(\|\hat{U}(0)\|_P+\|u\|_P\right)^2\leq 4 \norm{P}$, and condition (vi) holds with $L=5\norm{P}$ (we assume $P$ is not the zero matrix, otherwise the result trivially holds).\\
   
   We now verify that
   \[
   \theta(k,b):=\left\lceil(k+T(0)+1)\mathrm{e}^{\lambda b}\right\rceil, \quad \chi(\eps):=\left\lceil \frac{1}{\lambda^2\eps}\right\rceil,\ \text{ and }\ \Gamma:=\frac{2}{\lambda^2}
   \]
   satisfy condition (v'). First, for all $N,k\in\NN$
   \begin{equation*}
       \sum_{n=k}^N\gamma_n=\sum_{n=k}^N \frac{1}{\lambda(n+1)+T(0)}\geq \frac{1}{\lambda}\sum_{n=k}^N\frac{1}{n+1+T(0)}\geq \frac{1}{\lambda}\log\left(\frac{N+T(0)+2}{k+T(0)+1}\right),
   \end{equation*}
   which entails that $\theta$ is a rate of divergence for $\sum \gamma_n=\infty$. Second, for all $k\geq 1$
   \begin{equation*}
       \sum_{n=k}^\infty \gamma_n^2 \leq \frac{1}{\lambda^2}\sum_{n=k}^\infty \frac{1}{(n+1)^2}\leq \frac{1}{\lambda^2 k}.
   \end{equation*}
   Hence, $\chi$ is a rate of convergence for $\sum \gamma_n^2$. Moreover, we have
   \[
    \sum_{n=0}^\infty \gamma_n^2= \gamma_0^2 + \sum_{n=1}^\infty \gamma_n^2\leq \frac{1}{(\lambda + T(0))^2} + \frac{1}{\lambda^2} <\frac{2}{\lambda^2}=:\Gamma.
   \]
   The conclusion now follows from Theorem~\ref{res:RM:quant} after a straightforward computation of the rate $\rho$.
\end{proof}
Furthermore, if the coefficient $A$ is strictly bounded below by $\lambda/2$, we can use Theorem~\ref{res:RM:quant:fast} to obtain fast rates.
\begin{theorem}\label{Lemma-urns-fast}
 Let $P$ be a symmetric, positive-semi-definite $d \times d$ matrix. Assume that for all $n\in\NN$
  \begin{equation*}
    \langle \hat{U}(n)(R-\lambda I),\hat{U}(n) -u\rangle_P \le-A\norm{\hat{U}(n) -u}_P^2,\ \text{ for some }\ u \in \mathcal{S} \mbox{ and } A > \frac{\lambda}{2}.\tag{$**$}
\end{equation*}
 Then
  \[
\EE\left[\norm{\hat{U}(n)-u}_P^2\right]\leq\frac{\lambda t}{\lambda(n+1)+T(0)} \text{ and }\PP\left(\exists m\geq n\left(\norm{\hat{U}(n)-u}_P\geq \varepsilon\right)\right)\leq \frac{\lambda}{\varepsilon^2}\cdot\frac{t+2q}{\lambda(n+1)+T(0)}
\]
for all $n\in\mathbb{N}$ and $\varepsilon>0$, where $t\geq \max\left\{\lambda q/(2A-\lambda),5\norm{P}(1+T(0)/\lambda)\right\}$,
$q:=\left(\frac{D+\lambda\sqrt{\norm{P}}}{\lambda}\right)^2$ and $D\ge \max_{i\le d-1}\norm{R_i}_P$.
\end{theorem}

\begin{proof}
From the proof of Theorem~\ref{Lemma-urns-slow}, we just need to verify that under the assumption that $A>\frac{\lambda}{2}$, condition (v'') holds with $K:=(D+\lambda\sqrt{\norm{P}})^2$. Namely, we will show that it holds for 
   \[
    q:=\left(\frac{D+\lambda\sqrt{\norm{P}}}{\lambda}\right)^2> 0, \quad r:=1+\frac{T(0)}{\lambda}>1, \quad c:=\frac{2A}{\lambda}>1.
   \]
   Indeed, we have
   \begin{equation*}
       K\gamma_n^2=(D+\lambda\sqrt{\norm{P}})^2\frac{1}{(\lambda(n+1)+T(0))^2}=\frac{\left(\frac{D+\lambda\sqrt{\norm{P}}}{\lambda}\right)^2}{(n+(1+\frac{T(0)}{\lambda}))^2}=\frac{q}{(n+r)^2}
   \end{equation*}
   and
   \begin{equation*}
       2A\gamma_n=\frac{2A}{\lambda}\frac{1}{n+(1+\frac{T(0)}{\lambda})}=\frac{c}{n+r}.
   \end{equation*}
   The result now follows from Theorem~\ref{res:RM:quant:fast} with simple computations.
\end{proof}
For $1\le p \le d$, write $J_p$ for the $d \times d$ matrix formed by taking the first $p$ columns of the identity matrix and setting the remaining columns to 0, i.e.\ $\mathrm{diag}(1,\ldots,1,0,\ldots,0)$. Our quantitative results on the asymptotic distribution of the outdegrees in the random tree models considered in this paper will be consequences of Theorem \ref{Lemma-urns-fast}, with  $P$ set to be $J_p$ for suitable $p$. Here and throughout the paper, we write $\langle x,y \rangle_p:= \langle x,y \rangle_{J_p}$ and $\norm{x}_p:= \norm{x}_{J_p}$ for any  $x,y \in \RR^d$. If $x:=(x_0,\ldots,x_{d-1}) \in \RR^d$ and $1\le p \le d$ we write $x|_p$ for the truncated vector $(x_0,\ldots,x_{p-1}) \in \RR^p$. One can easily verify that $\langle x,y \rangle_p:=\langle x|_p,y|_p \rangle$ which entails $\norm{x}_p:= \norm{x|_p}$.\\

We have the immediate corollary of Theorem \ref{Lemma-urns-fast}, for the case $P=J_p$.

\begin{corollary}\label{Lemma-urns-fast2}
 Let  $1\le p \le d$. Assume that for all $n\in\NN$ 
  \begin{equation*}
    \langle \hat{U}(n)(R-\lambda I),\hat{U}(n) -u\rangle_p \le-A\norm{\hat{U}(n) -u}_p^2,\ \text{ for some }\ u \in \mathcal{S} \mbox{ and } A > \frac{\lambda}{2}.
\end{equation*}
 Then
  \[
\EE\left[\norm{\hat{U}(n)-u}_p^2\right]\leq\frac{\lambda t}{\lambda(n+1)+T(0)} \text{ and }\PP\left(\exists m\geq n\left(\norm{\hat{U}(n)-u}_p\geq \varepsilon\right)\right)\leq \frac{\lambda}{\varepsilon^2}\cdot\frac{t+2q}{\lambda(n+1)+T(0)}
\]
for all $n\in\mathbb{N}$ and $\varepsilon>0$, where $t\geq \max\left\{\lambda q/(2A-\lambda),5(1+T(0)/\lambda)\right\}$,
$q:=\left(\frac{D+\lambda}{\lambda}\right)^2$ and $D\ge \max_{i\le d-1}\norm{R_i}_p$.
\end{corollary}
\begin{proof}
   The proof follows from Theorem \ref{Lemma-urns-fast} after noting that $J_p$ is symmetric, positive-semi-definite and $\norm{J_p}=1$.
\end{proof}

\begin{remark}
    \label{rem:A:bound}
   All our quantitative results for random trees, which we shall present in the following section, will be consequences of  Corollary \ref{Lemma-urns-fast2}. In particular, in these applications, we take $p \in \{d,d-1\}$ and establish a condition stronger than the premise of the 
theorem requires. Concretely, in each case we exhibit $A > \lambda/2$ such that
\[
\langle w, w(R - \lambda I)\rangle_p \leq -A\|w\|_p^2 
\quad \text{for all } w \in \mathbb{R}^{d} \text{ with } 
\sum_{i=0}^{d-1} w_i = 0. \tag{***}
\]
Applying this with $w = \hat{U}(n) - u$, where $u$ is an eigenvector of $R$ 
corresponding to the eigenvalue $\lambda$ whose components sum to $1$, yields 
the premise of Corollary \ref{Lemma-urns-fast2}; note that the components of $w$ 
sum to $0$ since the components of both $u$ and $\hat{U}(n)$ sum to $1$.
\end{remark}

Such an $A$ satisfying $(***)$ can always be found.
\begin{proposition}
\label{prop:la}
Suppose $p \in \{d,d-1\}$. Then
    \[
\langle w, w(R - \lambda I)\rangle_p \leq -A|w\|_p^2 
\quad \text{for all } w \in \mathbb{R}^{d} \text{ with } 
\sum_{i=0}^{d-1} w_i = 0.
\]
for 
\[
A:= -\lambda_{\max}(\operatorname{symm}(C_p J_p
(R^\top - \lambda I)C_p^\top)).
\]
Where $C_p$ is a $p \times d$ matrix, defined via
\begin{equation*}
    C_{p} := \begin{cases}
        I & \text{if }\ p =d\\[0.5cm]
	(I_{d-1},\ (-1,\ldots,-1)^\top) & \text{if }\ p=d-1    
\end{cases}
\end{equation*}
and $I_{d-1}$ denotes the $(d-1) \times (d-1)$ identity matrix.
\end{proposition}

\begin{proof}
     Fix 
$w \in \mathbb{R}^d$ with components summing to $0$. If $\|w\|_p = 0$ 
then the result trivially follows, so assume $\|w\|_p \neq 0$. One verifies that $(w|_p)C_p = w$ in 
both cases: the case $p = d$ is trivial, and the case $p = d-1$ follows from 
the fact that the components of $w$ sum to $0$. Hence, for all such $w$,
\begin{align*}
    \langle w, w(R - \lambda I)\rangle_p 
    &= \langle wJ_p,\, w(R - \lambda I)\rangle 
    = wJ_p(R^\top - \lambda I)w^\top 
    = (w|_p)\bigl(C_p J_p(R^\top - \lambda I)C_p^\top\bigr)(w|_p)^\top \\
    &= \|w\|_p^2 \cdot 
    \frac{(w|_p)}{\|(w|_p)\|}
    \bigl(C_p J_p(R^\top - \lambda I)C_p^\top\bigr)
    \frac{(w|_p)^\top}{\|(w|_p)\|}.
\end{align*}
It therefore remains to maximise
\[
\frac{(w|_p)}{\|(w|_p)\|}
\bigl(C_p J_p(R^\top - \lambda I)C_p^\top\bigr)
\frac{(w|_p)^\top}{\|(w|_p)\|}
\]
over all $w$ with components summing to $0$. This is equivalent to maximising
\[
v\,\bigl(C_p J_p(R^\top - \lambda I)C_p^\top\bigr)v^\top
\]
over all $v \in \mathbb{R}^p$ with $\|v\| = 1$, which equals 
$\lambda_{\max}(\operatorname{symm}(C_p J_p(R^\top - \lambda I)C_p^\top))$, 
the largest eigenvalue of the symmetric part of 
$C_p J_p(R^\top - \lambda I)C_p^\top$ and the result follows.
\end{proof}
Thus, for $p \in \{d,d-1\}$, if 
\[
\lambda_{\max}\bigl(\operatorname{symm}(C_p J_p(R^\top - \lambda I)
C_p^\top)\bigr)<0,
\]
One is able to use Theorem \ref{Lemma-urns-slow} or Theorem \ref{Lemma-urns-fast} to obtain rates for the convergence of  $\hat{U}(n)$ both almost surely and in $L^2$. In particular, we have Theorem \ref{thrm:main:app}.
\begin{proof}[Proof of Theorem \ref{thrm:main:app}]
Proposition \ref{prop:la}, with $p=d$, implies
 \[
\langle w, w(R - \lambda I)\rangle_p \leq \lambda_{\max}(\operatorname{symm}(
(R - \lambda I)))|w\|_p^2 
\quad \text{for all } w \in \mathbb{R}^{d} \text{ with } 
\sum_{i=0}^{d-1} w_i = 0.
\]
Thus, by Remark \ref{rem:A:bound}, the premise of Corollary \ref{Lemma-urns-fast2} is satisfied with 
\[
A:= -\lambda_{\max}(\operatorname{symm}(
(R - \lambda I)))
\]
with the condition that $A> \lambda/2$ coming from the assumption of the theorem.
    
\end{proof}

\begin{remark}
If one can show that
\[
\lambda_{\max}(\operatorname{symm}(C_p J_p (R^\top - \lambda I) C_p^\top)) < 0 
 \text{ for } p = d \mbox{ or } d-1
\]
for all tenable, irreducible generalized P\'olya urns, then not only would 
Theorem~\ref{urns-convergence} follow, but one would also obtain quantitative 
rates for the almost sure and $L^2$ convergence (in the case $p=d-1$, one can conclude 
the $L^2$ and almost sure convergence of $\hat{U}(n)$ to $u$ with respect to 
the standard norm from the respective convergence result with respect to $\|\cdot\|_{d-1}$, as the components of $\hat{U}(n)$ and $u$ sum to 1). This question remains open.
\end{remark}

In our applications, rather than computing,
\[
\lambda_{\max}\bigl(\operatorname{symm}(C_p J_p(R^\top - \lambda I)
C_p^\top)\bigr),
\]
directly, we find a suitable $A$, satisfying condition $(***)$, via algebraic 
manipulations.

\section{Proof of quantitative results on random trees.}

\subsection{Proof of Theorems \ref{thrm:main:rec} }
Following the proof of \cite[Theorem 1.1]{janson2005asymptotic}, fix $M \geq 1$ 
and define
\[
U_i(n) := \begin{cases} X_i(n) & 0 \leq i \leq M-1, \\ 
\displaystyle\sum_{j \geq m} X_j(n) & i = M. \end{cases}
\]
Taking notation as in Section \ref{subsec:urns}, $(U(n))_{n \geq 0}$ is a 
balanced $(M+1)$-colour generalized P\'{o}lya urn with replacement matrix entries
$R_{ij} = -\delta_{i,j} + \delta_{i+1,j} + \delta_{0,j}$ for $i \leq M-1$ 
and $R_{Mj} = \delta_{0,j}$. The total number of balls added at each step is 
$\lambda = 1$. The matrix $R$ is irreducible (a fact we shall not use, in addition to the fact that the urn is tenable) and, 
as verified in \cite{janson2005asymptotic}, the vector 
$u = (1/2, 1/4, \ldots, 2^{-M}, 2^{-M})$ is the unique left eigenvector of $R$ 
with eigenvalue $\lambda = 1$ whose components are positive and sum to 1; 
existence and uniqueness are guaranteed by Perron--Frobenius theory. Furthermore, we have $T(n) = n+1$.\\

For all $i < M$, we have the following,
\begin{align*}
    \EE\left[\left|\frac{X_i(n)}{n} - \frac{1}{2^{i+1}}\right|^2\right]
    &= \EE\left[\left|\frac{X_i(n)}{n+1} -\frac{1}{2^{i+1}} 
        + \frac{X_i(n)}{n(n+1)}\right|^2\right] \\
    &\leq 2\EE\left[\left|\hat{U}_i(n) - \frac{1}{2^{i+1}}\right|^2\right] 
        + 2\EE\left[\left|\frac{\hat{U}_i(n)}{n}\right|^2\right] \\
    &\leq 2\EE\left[\left\|\hat{U}(n) - u\right\|_M^2\right] + \frac{2}{n^2},
\end{align*}
and
\begin{align*}
  &\PP\left(\exists\, m \geq n \left(\left|\frac{X_i(m)}{m} 
      - \frac{1}{2^{i+1}}\right| \geq \varepsilon\right)\right)\\
  &\quad\leq \PP\left(\exists\, m \geq n \left(\left|\hat{U}_i(m) 
      - \frac{1}{2^{i+1}}\right| + \left|\frac{\hat{U}_i(m)}{m}\right| 
      \geq \varepsilon\right)\right) \\
  &\quad\leq \PP\left(\exists\, m \geq n \left(\left|\hat{U}_i(m) 
      - \frac{1}{2^{i+1}}\right| \geq \frac{\varepsilon}{2}\right)\right) 
    + \PP\left(\exists\, m \geq n \left(\frac{\hat{U}_i(m)}{m}
      \geq \frac{\varepsilon}{2}\right)\right) \\
  &\quad\leq \PP\left(\exists\, m \geq n \left(
      \left\|\hat{U}(m) - u\right\|_M \geq \frac{\varepsilon}{2}\right)\right) 
    + \PP\left(\exists\, m \geq n \left(\frac{1}{m} \geq \frac{\varepsilon}{2}
      \right)\right) \\
  &\quad= \PP\left(\exists\, m \geq n \left(
      \left\|\hat{U}(m) - u\right\|_M \geq \frac{\varepsilon}{2}\right)\right),
\end{align*}
where the final equality holds whenever $n > 2/\varepsilon$ and with $\|\cdot\|_M$ denotes the norm in $\RR^M$ as before. Thus, to prove Theorem \ref{thrm:main:rec}, it therefore suffices 
to bound
\[
\EE\left[\left\|\hat{U}(n) - u\right\|_M^2\right] \quad \text{and} \quad 
\PP\left(\exists\, m \geq n \left(\left\|\hat{U}(m) - u\right\|_M 
\geq \frac{\varepsilon}{2}\right)\right).
\]
We do this by applying Corollary \ref{Lemma-urns-fast2}. We first have the following lemma.
\begin{lemma}
\label{lem:rec:ineq}
    For all $w \in \RR^{M+1}$ with $\sum_{i=0}^{M}w_i =0$, we have 
    \[
    \langle w, w(R- I)\rangle_M \le -\|w\|_M^2
    \]
\end{lemma}
\begin{proof}
    We have $(R-I)_{ij} = -2\delta_{i,j} +\delta_{i+1,j}+\delta_{0,j}$ for  $i\le M-1$ and $(R-I)_{Mj}=\delta_{0,j} -\delta_{M,j}$. Thus, we have 

    \[
    (w(R- I))_0 = \sum_{i=0}^{M-1}(-2w_i\delta_{i,0} + w_i\delta_{i+1,0}+w_i\delta_{0,0}) +w_M = -w_0+\sum_{i=1}^Mw_i =-2w_0
    \]
    where we use the fact that the components of $w$ sum to 0 to obtain the last equality. Now for $1\le j\le M-1$ we have 
    \[
    (w(R- I))_j = \sum_{i=0}^{M-1}(-2w_i\delta_{i,j} + w_i\delta_{i+1,j}+w_i\delta_{0,j})  = -2w_j+w_{j-1}.
    \]
    So we have 
    \[
    \langle w, w(R- I)\rangle_M = -2w_0^2 -2 \sum_{j=1}^{M-1}w_j^2 +  \sum_{j=1}^{M-1}w_{j}w_{j-1}. 
    \]
Thus, applying the AM-GM inequality, $ab \leq \frac{1}{2}(a^2 + b^2)$, to each product $w_jw_{j-1}$, yields,
\[
w_jw_{j-1} \leq \frac{1}{2}(w_j^2 + w_{j-1}^2), \quad 1 \leq j \leq M-1.
\]
Thus, we have 
\[
 \sum_{j=0}^{M}w_{j}w_{j-1} \leq \frac{1}{2}\left(w_0^2 + 2\sum_{i=1}^{M-2}w_j^2 + w_{M-1}^2\right).
\]
This implies 
\begin{align*}
\langle w, w(R- I)\rangle_M \leq -\frac{3}{2}w_0^2 - \sum_{j=1}^{M-2}w_j^2  -\frac{3}{2}w_{M-1}^2\le-\|w\|_M^2
\end{align*}
\end{proof}
\begin{remark}
In \cite[Theorem 1.3]{bjornberg2025random}, the authors obtain the asymptotic distribution of outdegrees in a variant of the random recursive tree in which a \emph{doubling event} occurs whenever the root is selected, showing that, almost surely, the limiting outdegree distribution for this model coincides with that of the ordinary random recursive tree, namely $u$. To this end, they define $U(n)$ analogously to this section, but for the doubling model, and show that  $(\hat{U}(n))$ satisfies the following stochastic approximation recurrence:
\[
    \hat{U}(n+1) = \hat{U}(n) + \gamma_n\bigl(F(\hat{U}(n)) + \Delta M_{n+1} 
    + \eta_{n+1}\bigr),
\]
where $(\Delta M_{n+1})$ is a martingale difference sequence, $(\eta_{n+1})$ is an error term, $(\gamma_n)$ is now \emph{stochastic}, and $F(x) = x(R - \lambda I) = x(R - I)$, with $R$ and $\lambda = 1$ the replacement matrix and balance of the random recursive tree, respectively. They then apply standard techniques from stochastic approximation theory, such as the Robbins--Siegmund theorem \cite{robbins-siegmund:71:lemma}. To apply these techniques, they must show, as part of their proof, that
\[
\langle x - u, F(x) \rangle < 0 \quad \text{for all } x \in \mathcal{S}
\]
(cf.\ condition $(*)$ in Section~\ref{subsec:urns}); to this end, they establish
\begin{equation}
\label{eq:double}
    \langle w, F(w)\rangle \leq -2w_0^2 
\quad \text{for all } w \in \mathbb{R}^{M+1} \text{ with } 
\sum_{i=0}^{M} w_i = 0,
\end{equation}
which suffices. This bound is not directly suitable for our purposes; however, we observe that a refinement of the calculation used to establish \eqref{eq:double} in the proof of \cite[Theorem~1.3]{bjornberg2025random} yields the `stronger' bound
\[
    \langle w, F(w)\rangle \leq -\frac{1}{2}\norm{w}^2 
\quad \text{for all } w \in \mathbb{R}^{M+1} \text{ with } 
\sum_{i=0}^{M} w_i = 0.
\]
However, this bound alone would not allow us to apply Corollary \ref{Lemma-urns-fast2}, since it gives only $A = \lambda/2$ rather than $A > \lambda/2$; it would allow only the application of Theorem \ref{Lemma-urns-slow}, resulting in a worse rate. Working instead with $\langle w, F(w)\rangle_M$, as we do in Lemma \ref{lem:rec:ineq}, allows one to obtain the sharper bounds of Corollary \ref{Lemma-urns-fast2}.

We conclude by noting that the quantitative results for stochastic approximation algorithms discussed here are not directly applicable to the recursive tree model of \cite{bjornberg2025random}, since the step sizes there are stochastic. We anticipate that the methods of \cite{NeriPowell2026,mythesis} could yield quantitative results for this model, and we leave this for future investigation.
\end{remark}
Now we have 
   \[
   \sum_{j=0}^{M} R_{0j}^2= \sum_{j=0}^{M} (-\delta_{0,j} + \delta_{1,j} + \delta_{0,j})^2= 1.
   \]
    \[
   \sum_{j=0}^{M} R_{Mj}^2= \sum_{j=0}^{d-1} \delta_{0,j}^2= 1.
   \]
and, for $0 < i \le M-1$
  \[
   \sum_{j=0}^{M} R_{ij}^2= \sum_{j=0}^{M} (-\delta_{i,j} + \delta_{i+1,j} + \delta_{0,j})^2= 3.
   \]
So, $\max_{i\le M} \norm{R_i}=\sqrt{3}\leq 2$. Thus, applying Corollary \ref{Lemma-urns-fast2} with $A :=1$ (by Lemma \ref{lem:rec:ineq}; see Remark \ref{rem:A:bound}), $p=M$ and $D:=2$ yields,
\[
\EE\bigl[\|\hat{U}(n) - u\|_M^2\bigr] \leq \frac{9}{n+2}
 \text{ and } 
\PP\left(\exists\, m \geq n \left(\|\hat{U}(m) - u\|_M \geq \varepsilon\right)\right) 
\leq  \frac{27}{\varepsilon^2(n+2)}.
\]
Hence, for all $i < M$,
\begin{align*}
    \EE\left[\left|\frac{X_i(n)}{n} - \frac{1}{2^i}\right|^2\right]
    &\leq \frac{18}{n+2} + \frac{2}{n^2}\le \frac{20}{n},
\end{align*}
and, for $n > 2/\varepsilon$,
\begin{align*}
  \PP\left(\exists\, m \geq n \left(\left|\frac{X_i(m)}{m} 
  - \frac{1}{2^i}\right| \geq \varepsilon\right)\right) 
  \leq \frac{108}{\varepsilon^2n}.
\end{align*}
This completes the proof of Theorem \ref{thrm:main:rec} where, for a fixed $i$, we take $M:=i+1$.

\subsection{Proof of Theorem \ref{thrm:main:plane}}
Following the proof of Theorem 1.3 of \cite{janson2005asymptotic}, fix $M \geq 1$ 
and define
\[
U_i(n) := \begin{cases}(i+1) Y_i(n) & 0 \leq i \leq m-1, \\ \displaystyle\sum_{j \geq m}(j+1) Y_j(n) & i = M, \end{cases}
\]
Then  $(U(n))_{n \geq 0}$ is a 
balanced $(M+1)$-colour generalized P\'{o}lya urn with replacement matrix entries
$R_{ij} = -(i+1)\delta_{i,j} +(i+2)\delta_{i+1,j}+\delta_{0,j}$ for $i \leq M-1$ 
and $R_{Mj}=\delta_{0,j}+\delta_{M,j}$ (which, as in the recursive tree case, is irreducible and the urn is tenable). Moreover, as shown in \cite{janson2005asymptotic}, one can verify that $u$, defined as 
 \[
u_i := \frac{2}{(i+2)(i+3)},\mbox{ for }  0 \leq i \leq M-1,\mbox{ and } u_M := \frac{2}{M+2},
\]
is an eigenvector, with eigenvalue $\lambda = 2$, which is the change in the number of balls at each step of the urn. Furthermore, we have $T(n)=2n+1$.\\

For all $i<M$, we have
\begin{align*}
    \EE&\left[ \left| \frac{Y_i(n)}{n} - \frac{4}{(i+1)(i+2)(i+3)} \right|^2 \right]= \EE\left[ \left| \frac{2}{i+1}\left( \frac{(i+1)Y_i(n)}{2n} - \frac{2}{(i+2)(i+3)} \right) \right|^2 \right]\\
    &\qquad = \EE\left[ \left| \frac{2}{i+1}\left( \frac{(i+1)Y_i(n)}{2n+1} - \frac{2}{(i+2)(i+3)} + \frac{(i+1)Y_i(n)}{2n(2n+1)} \right) \right|^2 \right]\\
    &\qquad = \EE\left[ \left| \frac{2}{i+1}\left( \hat{U}_i(n) - \frac{2}{(i+2)(i+3)} \right) + \frac{1}{i+1}\frac{\hat{U}_i(n)}{n} \right|^2 \right]\\
    &\qquad \leq 2\EE\left[ \left| \frac{2}{i+1}\left( \hat{U}_i(n) - \frac{2}{(i+2)(i+3)}\right) \right|^2 \right] + 2\EE\left[ \left| \frac{\hat{U}_i(n)}{(i+1)n} \right|^2 \right]\\
    &\qquad \leq\frac{8}{(i+1)^2}\EE\left[ \left\| \hat{U}(n) - u \right\|_M^2 \right] + \frac{2}{(i+1)^2n^2}
\end{align*}

and

\begin{align*}
    \PP&\left( \exists m\geq  n \left( \left|\frac{Y_i(m)}{m} - \frac{4}{(i+1)(i+2)(i+3)}\right|\geq \eps \right) \right)\\
    &\quad \leq \PP \left( \exists m\geq n \left( \frac{2}{i+1}\left| \hat{U}_i(m)-\frac{2}{(i+2)(i+3)}\right| + \left|\frac{\hat{U}_i(m)}{(i+1)m}\right|\geq \eps \right) \right)\\
    &\quad \leq \PP \left( \exists m\geq n \left( \frac{2}{i+1}\left| \hat{U}_i(m)-\frac{2}{(i+2)(i+3)}\right|\geq \frac{\eps}{2}\right)\right) + \PP\left( \exists m\geq n \left( \frac{\hat{U}_i(m)}{(i+1)m}\geq \frac{\eps}{2} \right) \right)\\
    &\quad \leq \PP \left( \exists m\geq n \left( \frac{2}{i+1}\left\| \hat{U}(m)-\frac{2}{(i+2)(i+3)}\right\|_M\geq \frac{\eps}{2}\right)\right) + \PP\left( \exists m\geq n \left( \frac{1}{m}\geq \frac{\eps}{2} \right) \right)\\
    &\quad = \PP \left( \exists m\geq n \left( \left\| \hat{U}(m)-\frac{2}{(i+2)(i+3)}\right\|_M\geq \frac{\eps(i+1)}{4}\right) \right),
\end{align*}
whenever $n>2/\eps$ and with $\|\cdot\|_M$ denotes the norm in $\RR^M$ as before. Thus, as in the proof of Theorem \ref{thrm:main:rec}, the result will follow from computing an upper bound on both
\[
\EE\left[ \left\| \hat{U}(n) - u \right\|_M^2 \right]\ \text{ and }\ \PP \left( \exists m\geq n \left( \left\| \hat{U}(m)-u\right\|_M\geq \frac{\eps(i+1)}{4}\right) \right).
\]

We have the following lemma, akin to Lemma \ref{lem:rec:ineq}.
\begin{lemma}\label{lemma_plane_trees}
    For all $w \in \RR^{M+1}$ with $\sum_{i=0}^Mw_i =0$, we have 
    \[
    \langle w, w(R- \lambda I)\rangle_M \le -\frac{3}{2}\norm{w}_M
    \]
\end{lemma}
\begin{proof}
    We have $(R-\lambda I)_{ij} = -(i+3)\delta_{i,j} +(i+2)\delta_{i+1,j}+\delta_{0,j}$ for  $i\le M-1$ and also $(R-\lambda I)_{Mj}=\delta_{0,j}-\delta_{M,j}$.
 Thus, 
 \begin{align*}
       (w(R- \lambda I))_0& =\sum_{i=0}^{M-1}(-(i+3)\delta_{i,0}w_i +(i+2)\delta_{i+1,0}w_i+\delta_{0,0}w_i) +w_M\\& = -2w_0+\sum_{i=1}^Mw_i =-3w_0
 \end{align*}
    where we use the fact that the components of $w$ sum to 0 to obtain the last equality. Now for $1\le j\le M-1$ we have 
    \[
    (w(R- I))_j = \sum_{i=0}^{M-1}(-(i+3)\delta_{i,j}w_i +(i+2)\delta_{i+1,j}w_i+\delta_{0,j}w_i)   = -(j+3)w_j+(j+1)w_{j-1}
    \]
    and we have 
    So we have 
    \[
    \langle w, w(R- \lambda I)\rangle_M =-3w_0^2 - \sum_{j=1}^{M-1}(j+3)w_j^2 +  \sum_{j=1}^{M-1}(j+1)w_{j}w_{j-1}. 
    \]
Applying AM-GM yields,
\[
\sum_{j=1}^{M-1} (j+1)w_jw_{j-1} \leq \frac{1}{2}\left(2w_0^2 + \sum_{j=1}^{M-2}(2j+3)w_j^2 + Mw_{M-1}^2\right).
\]

\begin{align*}
\langle w, w(R- \lambda I)\rangle &\leq  -3w_0^2 + w_0^2 +\frac{1}{2}\sum_{j=1}^{M-2}(2j+3)w_j^2 - \sum_{j=1}^{M-1} (j+3)w_j^2+ \frac{M}{2}w_{M-1}^2 \\
&= -2w_0^2-\frac{3}{2} \sum_{j=1}^{M-2}w_j^2-(M+2)w_{M-1}^2+ \frac{M}{2}w_{M-1}^2\\
&= -2w_0^2-\frac{3}{2} \sum_{j=1}^{M-2}w_j^2-\frac{M+4}{2}w_{M-1}^2 .
\end{align*}
and the result follows.
\end{proof} 
We have
\[
\|R_0\|^2=\sum_{j=0}^M(-\delta_{0j} + 2\delta_{1j} + \delta_{0j})^2 =4, \quad \|R_M\|^2=\sum_{j=0}^M(\delta_{0j}+\delta_{Mj})^2=2,
\]
and for $0<i\leq M-1$
\begin{align*}
    \|R_i\|^2&=\sum_{j=0}^M(-(i+1)\delta_{0j} + (i+2)\delta_{i+1,j} +\delta_{0j})^2\\
    &=(i+2)^2+(i+1)^2+1\\
    &\leq 2(i+2)^2\leq 2(M+2)^2.
\end{align*}
Hence, we have $\max_{i\leq M}\|R_i\|_M\leq\max_{i\leq M}\|R_i\|\leq \sqrt{2}(M+2)\leq 2(M+2)$.\\

We can now apply Corollary \ref{Lemma-urns-fast2} with $A:=3/2$ (by Lemma \ref{lemma_plane_trees}; see Remark \ref{rem:A:bound}), $p=M$ and $D:=2(M+2)$, to obtain
\[
\EE\left[ \left\| \hat{U}(n)-u\right\|_M^2 \right]\leq \frac{4(M+3)^2}{2n+3}\ \text{ and }\ \PP \left( \exists m\geq n \left( \|\hat{U}(m)-u\|_M>\eps\right)\right) \leq \frac{1}{\eps^2}\cdot\frac{6(M+3)^2}{2n+3}  .
\]
Thus, for all $i<M$ and $n\geq 1$,
\[
\EE\left[ \left| \frac{Y_i(n)}{n} - \frac{4}{(i+1)(i+2)(i+3)} \right|^2 \right]\leq \frac{32}{2n+3}\left(\frac{M+3}{i+1}\right)^2 + \frac{2}{(i+1)^2n^2}\leq \frac{17}{n}\left(\frac{M+3}{i+1}\right)^2
\]
and, for $n>2/\eps$,
\[
\PP\left( \exists m\geq  n \left( \left|\frac{Y_i(m)}{m} - \frac{4}{(i+1)(i+2)(i+3)}\right|\geq \eps \right) \right)\leq \frac{48}{\eps^2n}\left(\frac{M+3}{i+1}\right)^2.
\]
This completes the proof of Theorem~\ref{thrm:main:plane} by taking $M:=i+1$ and noting that for all $i \ge 0$, we have 
\[
\left(\frac{i+4}{i+1}\right)^2= \left(1+\frac{3}{i+1}\right)^2\le 16.
\]

\subsection{Proof of Theorem \ref{thrm:main:plane}}
Again, following Janson (cf.\ \cite[Section 6]{janson2005asymptotic}), we define
\[
U_i(n) := (d-i)Z_i(n).
\]
Then  $(U(n))_{n \geq 0}$ is a balanced $d$-colour generalized P\'{o}lya urn with replacement matrix entries
$R_{ij} = d\delta_{0,j} -(d-i)\delta_{i,j}+(d-i-1)\delta_{i+1,j}$ for $i \leq d-1$. Moreover, as in \cite{janson2005asymptotic}, one can verify that $u$ defined as 
 \[
u_i := \frac{\binom{2d-i-2}{d-1}}{\binom{2d-1}{d}} = \frac{d-i}{d-1}\frac{\binom{2d-i-2}{d-2}}{\binom{2d-1}{d-1}} = \frac{d-i}{d-1}s_i
\]
is an eigenvector of $R$, with eigenvalue $\lambda = d-1$, which is the change in the number of balls at each step of the urn. Furthermore, we have $T(n)=(d-1)n+1$.\\

For all $i<d$, we have
\begin{align*}
    \EE&\left[ \left| \frac{Z_i(n)}{n} - s_i \right|^2 \right]= \EE\left[ \left| \frac{d-1}{d-i}\left( \frac{(d-i)Z_i(n)}{(d-1)n} - u_i \right) \right|^2 \right]\\
    &\qquad = \EE\left[ \left| \frac{d-1}{d-i}\left(\frac{(d-i)Z_i(n)}{(d-1)n+1} - u_i + \frac{(d-i)Z_i(n)}{(d-1)n((d-1)n+1)} \right) \right|^2 \right]\\
    &\qquad = \EE\left[ \left|\frac{d-1}{d-i}\left( \hat{U}_i(n) -u_i \right) + \frac{1}{d-i}\frac{\hat{U}_i(n)}{n} \right|^2 \right]\\
    &\qquad \leq 2\EE\left[ \left| \frac{d-1}{d-i}\left( \hat{U}_i(n) -u_i\right) \right|^2 \right] + 2\EE\left[ \left| \frac{\hat{U}_i(n)}{(d-i)n} \right|^2 \right]\\
    &\qquad \leq 2\left(d-1\right)^2\EE\left[ \left\| \hat{U}(n) - u \right\|^2 \right] + \frac{2}{n^2}
\end{align*}

and

\begin{align*}
    \PP&\left( \exists m\geq  n \left( \left|\frac{Z_i(m)}{m} - s_i \right|\geq \eps \right) \right)\\
    &\quad \leq \PP \left( \exists m\geq n \left( \frac{d-1}{d-i}\left| \hat{U}_i(m)-u_i\right| + \left|\frac{\hat{U}_i(m)}{(d-i)m}\right|\geq \eps \right) \right)\\
    &\quad \leq \PP \left( \exists m\geq n \left( \frac{d-1}{d-i}\left| \hat{U}_i(m)-u_i\right|\geq \frac{\eps}{2}\right)\right) + \PP\left( \exists m\geq n \left( \frac{\hat{U}_i(m)}{(d-i)m}\geq \frac{\eps}{2} \right) \right)\\
    &\quad \leq \PP \left( \exists m\geq n \left((d-1)\left\| \hat{U}(m)-u\right\|\geq \frac{\eps}{2}\right)\right) + \PP\left( \exists m\geq n \left( \frac{1}{m}\geq \frac{\eps}{2} \right) \right)\\
    &\quad = \PP \left( \exists m\geq n \left( \left\| \hat{U}(m)-u\right\|\geq \frac{\eps}{2(d-1)}\right) \right),
\end{align*}
whenever $n>2/\eps$. Thus, the result will follow from computing an upper bound on both
\[
\EE\left[ \left\| \hat{U}(n) - u \right\|^2 \right]\ \text{ and }\ \PP \left( \exists m\geq n \left( \left\| \hat{U}(m)-u\right\|\geq \frac{\eps}{2(d-1)}\right) \right).
\]

We have the following lemma.
\begin{lemma}\label{lemma_d_trees}
    For all $w \in \RR^{d}$ with $\sum_{i=0}^{d-1}w_i =0$, we have 
    \[
    \langle w, w(R- \lambda I)\rangle \le -\frac{2d-1}{2}\norm{w}
    \]
\end{lemma}
\begin{proof}
    We have 
    \[
    (R-\lambda I)_{ij} = -(2d-i-1)\delta_{i,j} +(d-i-1)\delta_{i+1,j}+d\delta_{0,j}.
    \]
 Thus, we have 
 \begin{align*}
       (w(R- \lambda I))_0& =\sum_{i=0}^{d-1}(-(2d-i-1)\delta_{i,0}w_i +(d-i-1)\delta_{i+1,0}w_i+d\delta_{0,0}w_i)\\
       &=-(2d-1)w_0 + d\sum_{i=0}^{d-1}w_i=-(2d-1)w_0
 \end{align*}
    where we use the fact that the components of $w$ sum to 0 to obtain the last equality. Now for $1\le j\le d-1$ we have 
    \[
    (w(R- I))_j = \sum_{i=0}^{d-1}(-(2d-i-1)\delta_{i,j}w_i +(d-i-1)\delta_{i+1,j}w_i+d\delta_{0,j}w_i)   = -(2d-j-1)w_j+(d-j)w_{j-1}
    \]
    So we have 
    \[
    \langle w, w(R- \lambda I)\rangle =-(2d-1)w_0^2 - \sum_{j=1}^{d-1}(2d-j-1)w_j^2 +  \sum_{j=1}^{d-1}(d-j)w_{j}w_{j-1}. 
    \]
Now applying AM-GM, yields
\[
\sum_{j=1}^{d-1} (d-j)w_jw_{j-1} \leq \frac{1}{2}\left((d-1)w_0^2 + \sum_{j=1}^{d-1}(2d-2j-1)w_j^2\right).
\]
Thus,
\begin{align*}
\langle w, w(R- \lambda I)\rangle &\leq -(2d-1)w_0^2 +\frac{d-1}{2}w_0^2 - \sum_{j=1}^{d-1}(2d-j-1)w_j^2 +  \sum_{j=1}^{d-1}\frac{2d-2j-1}{2}w_j^2 \\
&= -\frac{3d-1}{2}w_0^2 - \sum_{j=1}^{d-1}\frac{2d-1}{2}w_j^2  \le  -\frac{2d-1}{2}\norm{w}
\end{align*}
and the result follows.
\end{proof}
We have
\[
\|R_0\|^2=\sum_{j=0}^{d-1}(d\delta_{0,j} -d\delta_{0,j}+(d-1)\delta_{1,j})^2 =(d-1)^2
\]
and for $0<i\leq d-1$
\begin{align*}
    \|R_i\|^2&=\sum_{j=0}^{d-1}(d\delta_{0,j} -(d-i)\delta_{i,j}+(d-i-1)\delta_{i+1,j})^2\\
    &=d^2+(d-i)^2+(d-i-1)^2\\
    &\leq 3d^2.
\end{align*}
Hence, we have $\max_{i\leq d-1}\|R_i\|\leq\max_{i\leq d-1}\|R_i\|\leq \sqrt{3}d\leq 2d$.\\

We can now apply Corollary \ref{Lemma-urns-fast2} with $A:=(2d-1)/2$ (it is clear that $A> \lambda/2 =(d-1)/2$), $p=d$ and $D:=2d$, to obtain, 
\[
\EE\left[ \left\| \hat{U}(n)-u\right\|^2 \right]\leq \frac{(3d-1)^2}{d((d-1)n+d)}\ 
\]
and 
\[
\PP \left( \exists m\geq n \left( \|\hat{U}(m)-u\|>\eps\right)\right) \leq \frac{(3d-1)^3}{\eps^2d(d-1)}\cdot\frac{1}{(d-1)n+d}  .
\]
Thus, for all $i<d$ and $n\geq 1$,
\[
\EE\left[ \left| \frac{Z_i(n)}{n} - s_i\right|^2 \right]\leq \frac{2(d-1)^2(3d-1)^2}{d((d-1)n+d)} + \frac{2}{n^2}\leq \frac{2(3d-1)^2}{n}+ \frac{2}{n^2}\le \frac{19d^2}{n}
\]
and, for $n>2/\eps$,
\[
\PP\left( \exists m\geq  n \left( \left|\frac{Z_i(m)}{m} - s_i\right|\geq \eps \right) \right)\leq \frac{4(3d-1)^3(d-1)^2}{\eps^2d(d-1)}\cdot\frac{1}{(d-1)n+d}\le \frac{12(3d-1)^2}{\eps^2n}\le  \frac{108d^2}{\eps^2n}.
\]
For the case $i = d$, we have
\[
\EE\left[ \left| \frac{Z_d(n)}{n} - s_d\right|^2 \right]= \EE\left[ \left| \sum_{i=0}^{d-1}\frac{Z_i(n)}{n} - s_i\right|^2 \right]\le d \sum_{i=0}^{d-1}\EE\left[ \left| \frac{Z_i(n)}{n} - s_i\right|^2 \right] \le \frac{19d^4}{n}.
\]
The first equality follows from the fact that both the sum of $Z_i(n)/n$ and $s_i$ for $0\le i\le d$ are equal to 1. The penultimate inequality follows from Cauchy-Schwarz. Furthermore, we have
\begin{align*}
    \PP\left( \exists m\geq  n \left( \left|\frac{Z_d(m)}{m} - s_d\right|\geq \eps \right) \right) &=\PP\left( \exists m\geq  n \left( \left|\sum_{i=0}^{d-1}\frac{Z_i(n)}{n} - s_i\right|\geq \eps \right) \right)\\
    &\le \sum_{i=0}^{d-1}\PP\left( \exists m\geq  n \left( \left|\frac{Z_i(n)}{n} - s_i\right|\geq \frac{\eps}{d} \right) \right)\\
    &\le  \frac{108d^5}{\eps^2n}.
\end{align*}
This concludes the proof.

\noindent
{\bf Acknowledgments:} The authors would like to thank C\'ecile Mailler for her helpful comments. The second author was supported
by the DFG Project PI 2070/1-1 and by FCT -- Fundação para a Ciência e a Tecnologia, through national funds, under the project UID/04561/2025 (\protect{\url{https://doi.org/10.54499/UID/04561/2025}}).

\bibliographystyle{plain}
\bibliography{ref}

\end{document}